\documentclass{amsart}

\usepackage{amsmath}
\usepackage{amsthm}
\usepackage{amssymb}
\usepackage{delarray}
\usepackage[dvips]{graphics}
\usepackage{epsfig}
\usepackage{color}
\usepackage{hyperref}
\usepackage{amsmath}    
\usepackage{graphicx} 
\usepackage{verbatim} 
\usepackage{color} 
\usepackage{subfigure}  
\usepackage{hyperref}  
\usepackage{fullpage}
\usepackage{bbm}
\usepackage{dsfont} 
\usepackage{amssymb}
\usepackage{mathrsfs}
\usepackage{enumerate}
\usepackage[super]{nth}

\newcommand{\Z}{{\mathbb Z}}

\newcommand{\N}{{\mathbb N}}

\newcommand{\mR}{{\mathcal R}}

\DeclareMathOperator{\supp}{supp}

\newtheorem{question}{Question}

\newtheorem{thm}{Theorem}

\newtheorem{conj}{Conjecture}
\newtheorem{lemma}{Lemma}
\newtheorem{prop}{Proposition}
\newtheorem{cor}{Corollary}

\theoremstyle{definition}

\theoremstyle{remark}

\title[The Realization Problem for Delta Sets]{The Realization Problem for Delta Sets of Numerical Semigroups}
\author{Stefan Colton}
\address{Department of Mathematics, Yale University, New Haven, CT 06511}
\email{stefan.colton@yale.edu}

\author{Nathan Kaplan}
\address{Department of Mathematics, University of California, Irvine, CA 92697} 
\email{nckaplan@math.uci.edu}

\date{May, 1 2016}


\begin{document}

\maketitle

\begin{abstract} 
The delta set of a numerical semigroup $S$, denoted $\Delta(S)$, is a factorization invariant that measures the complexity of the sets of lengths of elements in $S$.  We study the following problem: Which finite sets occur as the delta set of a numerical semigroup $S$?  It is known that $\min \Delta(S) = \gcd \Delta(S)$ is a necessary condition.  For any two-element set $\{d,td\}$ we produce a semigroup $S$ with this delta set.  We then show that for $t\ge 2$, the set $\{d,td\}$ occurs as the delta set of some numerical semigroup of embedding dimension three if and only if $t=2$.
\end{abstract}

\section{Introduction}

There are a number of invariants that have been used to study the failure of unique factorization in commutative cancellative monoids.  Non-unique factorization in these monoids has received quite a bit of attention in the recent literature, for example see \cite{GHK} and the extensive list of references therein.  Numerical semigroups give a particularly concrete setting in which to study these factorization problems.  One motivation for studying the factorization theory of numerical semigroups comes from their associated numerical semigroup rings.  These rings often give concrete instances of more general problems in commutative algebra \cite{GSL}.

There are several factorization invariants of numerical semigroups and related commutative monoids that have been studied extensively in the recent literature, for example the maximal denumerant \cite{BHJ}, the catenary and tame degree \cite{BCRSS, CGL, Omi}, and the $\omega$-invariant \cite{ACKT, GMV2}.  In this paper we focus on another invariant, the delta set \cite{BCS, BCLMS, BCKR,CDHK, CGP, CHK, Chap, GMV, GSLM}. This set measures the complexity of sets of factorization lengths for elements of the semigroup.  The goal of studying these invariants is to understand when two semigroups have similar factorization behavior.  One idea behind the delta set is that in semigroups with similar factorization behavior the structure of the sets of lengths should be similar.  

Much effort has gone into computing invariants for certain classes of semigroups.  The following related question has received relatively less attention.  Given a value for a factorization invariant, does there exist a numerical semigroup realizing it?  We focus on a particular question that we refer to as the \emph{realization problem for delta sets of numerical semigroups}.
\begin{question}
\begin{enumerate}
\item Which finite sets $T$ occur as $\Delta(S)$ for some numerical semigroup $S$?
\item Given an integer $e \ge 2$, which finite sets $T$ occur as $\Delta(S)$ for some numerical semigroup $S$ with embedding dimension $e$?
\end{enumerate}
\end{question}
In this paper we show that any set $T = \{d,td\}$ with $d \ge 1$ and $t\ge 2$ has a positive answer to the first question by explicitly producing a semigroup $S$ with $\Delta(S) = \{d,td\}$.  Factorizations in semigroups with embedding dimension two are easy to understand, but several problems remain unsolved in the embedding dimension three case.  We show that if $\{d,td\}$ with $t \ge 2$ has a positive answer to the second question when $e = 3$, then $t = 2$.

We also carefully study the minimal presentations of the classes of semigroups proving the results stated above.  The minimal presentation is a  set of generators of a monoid associated to the semigroup that describes all possible ways of moving between factorizations of the same element.  These presentations are extremely useful in understanding factorization properties and have been thoroughly investigated \cite{GSO, Phil1, Phil2}.

\subsection{Background}

We recall that a \emph{numerical semigroup} is an additive submonoid of $\N = \{0,1,2,\ldots\}$ with finite complement.  Every numerical semigroup $S$ has a unique \emph{minimal generating set}, that is, there exists a set of minimum cardinality $\{n_1,\ldots, n_e\}$ of distinct positive integers such that 
\[ 
S = \{a_1 n_1 + \cdots + a_e n_e\ | a_i \in \N\}.
\]
The number of elements of a minimal generating set is called the \emph{embedding dimension} of $S$ and is usually denoted by $e$.  We write $S = \langle n_1, \ldots, n_e\rangle$ if $S$ has minimal generating set $\{n_1,\ldots, n_e\}$.

The \emph{factorization homomorphism} $\varphi \colon \N^e \rightarrow S$ is defined by 
\[
\varphi(a_1,\ldots, a_e) = a_1 n_1 + \cdots + a_e n_e.
\]
If $\varphi(a_1,\ldots, a_e) = x$ then we say that $(a_1,\ldots, a_e)$ is a \emph{factorization} of $x$.  The \emph{length} of this factorization is defined as $a_1 + \cdots + a_e$.  The set of factorizations of $x$ is $\varphi^{-1}(x)$, which is clearly finite.  Let $\mathcal{L}(x)$ denote the corresponding set of factorization lengths.

Suppose $\mathcal{L}(x) = \{\ell_1 < \ell_2 < \cdots < \ell_m\}$.  The \emph{delta set} of $x$ is the set of differences of consecutive elements in this list, 
\[
\Delta(x) = \{\ell_{i+1} - \ell_i\ |\ i\in [1,m-1]\}.
\]  
The delta set of $S$ is defined as $\Delta(S) = \bigcup_{x\in S} \Delta(x)$.  This set gives a measure of how far $S$ is from being a unique factorization domain.  In a unique factorization domain, each element of the domain has exactly one factorization.  In a half-factorial domain, factorizations are not unique but every factorization has the same length, so the delta set of each element is empty.  The delta set of $S$ consists of a single element $d$ if and only if at least one element has at least two factorization lengths and the set of lengths of every element is an arithmetic progression with common difference $d$.

We give an overview of previous results on delta sets.  It is known that $\Delta(S)$ is finite.  An explicit finite set that determines $\Delta(S)$ is given by the following result.
\begin{thm}[Corollary 3 in \cite{CHK}]
Let $S = \langle n_1,\ldots, n_e\rangle$ and $N = 2 e n_2 n_e^2 + n_1 n_e$.  Then 
\[
\Delta(S) = \bigcup_{x \in S \atop x \le N} \Delta(x).
\]
\end{thm}
This result shows how to determine $\Delta(S)$ in finite time.  It has subsequently been refined in Corollary 19 of \cite{GMV}. Many algorithms related to numerical semigroups have been implemented in the \texttt{numericalsgps} package for the computer algebra system GAP \cite{GAP}, and recently improvements have been suggested \cite{BOP, GMV}.   We have used data from this package extensively throughout this project.

In order for a finite set $T$ to occur as $\Delta(S)$ for a numerical semigroup $S$ the following necessary condition must be satisfied.
\begin{prop}[Proposition 1.4.4 in \cite{GHK}]\label{Geroldinger}
Let $S$ be a numerical semigroup.  Then $\min \Delta(S) = \gcd \Delta(S)$.
\end{prop}

The following result gives an easy way to compute this minimum value in terms of a minimal generating set.
\begin{prop}[Proposition 2.9 in \cite{BCKR}]\label{MinD}
Let $S = \langle n_1,\ldots, n_e\rangle$.  Then 
\[
\min \Delta(S) = \gcd\{n_{i+1} - n_i\ |\ i \in [1,e-1]\}.
\]
\end{prop}

There are not so many families of semigroups for which the delta set is known.  However, it is easy to see that every set consisting of a single element occurs as a delta set of a numerical semigroup of embedding dimension two. 
\begin{prop}
Let $S = \langle n_1, n_2\rangle$ with $n_1 < n_2$ satisfying $\gcd\{n_1, n_2\} = 1$.  Then $\Delta(S) = \{n_2 - n_1\}$.
\end{prop}

More generally, every set of the form $\{d,2d,\ldots, td\}$ is also known to occur as a delta set.
\begin{prop}[Corollary 4.8 in \cite{BCKR}]
Let $S = \langle n, n+d, (d+1)n-d\rangle$ with $n \ge 3,\ d\ge 1$ and $\gcd\{n,d\} = 1$.  Then 
\[
\Delta(S) = \left\{d,2d,\ldots, \left\lfloor\frac{n+d-1}{d+2} \right\rfloor d \right\}.
\]
\end{prop}

The delta sets of the above proposition begin with a minimum value $d$ and then contain all multiples of $d$ up to some maximum.  We call this an \emph{interval with difference $d$}.  It is more difficult to find classes of semigroups with delta sets not of this form.  In order to show that an integer $k$ is in the delta set of a semigroup $S$ we need only find an element $x\in S$ with $k \in \Delta(x)$.  Showing that $k \not\in \Delta(S)$ is generally much more challenging.  The explicit computation of the delta sets of the following family show that large `gaps' can occur within delta sets, that there are delta sets which are in some sense far from being intervals.
\begin{prop}[Proposition 4.9 in \cite{BCKR}]
Let $S = \langle n, n+1, n^2-n-1\rangle$ with $n\ge 3$.  Then
\[
\Delta(S) = [1,n-2] \cup \{2n-5\}.
\]
\end{prop}

Very little is known about sets that cannot occur as delta sets.  Given a semigroup $S$ we can consider factorizations with respect to a non-minimal generating set and give a corresponding definition of the delta set of $S$.  In this setting there is one main result relevant to the realization problem.
\begin{thm}[Theorem 3.12 in \cite{CDHK}]
Let $S = \langle n_1, n_2\rangle$ and let $s = i n_1 + j n_2$ with $j \ge 0$ and $0 \le i < n_2$.  If the delta set of $S$ with respect to the generating set $\{n_1, n_2, s\}$ is $\{1,t\}$ then $t = 2$.
\end{thm}
We note the similarity of this theorem to the main result of Section \ref{3gen}, however for the duration of the paper we only consider factorizations with respect to minimal generating sets.

Extensive computer calculations described in \cite{BCLMS} give examples of other sets that occur as delta sets.  For example, $\Delta(\langle 6,13,14,16\rangle) = \{1,3\}$. The initial motivation for this work was to understand whether this is one instance of a more general family of examples that have $\Delta(S) = \{1,t\}$ for larger values of $t$.  In the next section we give a construction of such a family.

The authors of \cite{BCLMS} conjecture that for any $t \ge 3,\ \{1,t\}$ cannot occur as the delta set of a semigroup of embedding dimension three.  More specifically, they make the following conjectures.
\begin{conj}[Conjectures 12.3 and 12.4 in \cite{BCLMS}]
\ \\
\begin{enumerate}
\item Let $S = \langle n_1, n_2, n_3\rangle,\ \gcd\{n_3-n_2, n_2-n_1\} = d$, and suppose that $|\Delta(S)| > 1$.  Then $2d \in \Delta(S)$.

\item Let $S = \langle n_1, n_2, n_3\rangle,\ \gcd\{n_3-n_2, n_2-n_1\} = d$, and suppose that $|\Delta(S)| > 2$. Then $3d \in \Delta(S)$.
\end{enumerate}
\end{conj}

In Section \ref{3gen} we prove a piece of the first part of this conjecture, that $\{d,td\}$ for $d\ge 1,\ t\ge 3$ does not occur as the delta set of an embedding dimension three numerical semigroup.  While this paper was being completed we discovered that these two conjectures have been proven in \cite{GSLM}.  This paper also gives another proof of one of our main results, Theorem \ref{no13}, but uses significantly different methods.

\section{A Family of Semigroups with Delta Sets of Size Two}

We begin this section by recalling the definition of a minimal presentation of a numerical semigroup $S$.  Informally, a minimal presentation consists of a minimal set of `trades' needed to go between any two factorizations of an element $x\in S$.  We describe this in more precise detail below using the notation of Chapter 5 of \cite{GSR2} and Section 1 of Chapter 7 of \cite{GSR}.  We then introduce an explicit family of semigroups and compute their minimal presentations.  Finally, we use these minimal presentations to show that these semigroups have delta sets of size two.

We closely follow the presentation of \cite{GSO}.  Let $S = \langle n_1,\ldots, n_e\rangle$ be a numerical semigroup of embedding dimension $e$ and recall the factorization homomorphism $\varphi \colon \N^e \rightarrow S$ given in the previous section.  The \emph{kernel congruence} of $\varphi,\ \sim$ is defined by $u \sim v$ if and only if $\varphi(u) = \varphi(v)$.  This is a congruence, meaning that it is an equivalence relation compatible with addition.  Given $\rho \subseteq \N^e \times \N^e$, the congruence generated by $\rho$ is the least congruence containing it.  We say that $\rho$ is a system of generators of $\sim$ if $\rho$ generates $\sim$ as a congruence.  A \emph{presentation} of a numerical semigroup $S$ is a system of generators of its kernel congruence. The presentation is \emph{minimal} if it is a minimal system of generators for this congruence.  See the discussion before and after Proposition 5.11 of \cite{GSR2} for precise definitions.  This result, along with Propositions 8.4 and 8.5 of \cite{GSR} give a concrete way to view these concepts.

\begin{thm}\label{MinPres}
Let 
\[
S = \langle p^x - 2, 2(p^x-2)+1, 2(p^x-2)+p,\ldots, 2(p^x-2)+p^{x-1} \rangle
\] 
where $p, x \ge 2$ and $(p,x) \neq (2,2)$.  Then a minimal presentation of $S$ has size $x+1$ and is given by the elements \small{
\[
((2p-3,2,0,\ldots, 0),(0,\ldots, 0,p)), ((2x(p-1)-1,0,\ldots, 0),(0,p-2,p-1,\ldots, p-1))
\]
}
and for each $i \in [1,x-1]$ 
\[
v_i := ((0,\ldots, 0,p,0,\ldots, 0),(2(p-1),0,\ldots, 0,1,0\ldots, 0)),
\]
where the entry $p$ on the left is in the $i$\textsuperscript{th} position, where we count starting at $0$, and the entry $1$ on the right is in position $i+1$.
\end{thm}
For the remainder of this section $S$ will denote the semigroup given in Theorem \ref{MinPres}.  

It is known that the minimal presentation of a numerical semigroup of embedding dimension $e$ has size at least $e-1$ and size at most $\frac{(2n_1 - e + 1)(e-2)}{2} +1$, where $n_1$ is the smallest nonzero element of $S$.  The lower bound is Theorem 9.6 in \cite{GSR} and the upper bound is Theorem 8.26 in the same reference.  The semigroups for which the lower bound is an equality, for example those of embedding dimension two, are known as \emph{complete intersection numerical semigroups}.  This class of semigroups has received considerable attention in recent years \cite{AGS, DMS, GSL}, in part because of connections to commutative algebra and algebraic geometry.  We note that the semigroups of Theorem \ref{MinPres} are not complete intersections, but they have minimal presentations of cardinality equal to the embedding dimension, exactly one greater than the lower bound.  

Let $S$ be a numerical semigroup of embedding dimension $e$ and $A = (A_1,\ldots, A_e)$ be a factorization of an element $n\in S$.  Recall that $\varphi^{-1}(n)$ is the set of factorizations of $n$.  The \emph{support} of $A$, denoted $\supp(A)$, is the set of $i$ such that $A_i \neq 0$.  Let $B$ be another factorization of $n$.  Then $A$ and $B$ are in the same \emph{$\mR$-class} if there exists a chain of distinct factorizations $x_0, x_1, \ldots, x_k \in \varphi^{-1}(n)$ such that $x_0 = A, x_k = B$ and for each $i \in [0,k-1],\ \supp(x_i) \cap \supp(x_{i+1}) \neq \emptyset$.  This relation partitions the set $\varphi^{-1}(n)$.  The \emph{Betti elements} are the elements $n\in S$ such that $\varphi^{-1}(n)$ has more than one $\mR$-class.  Since $S$ is finitely presented, this set is finite.

We recall some notation from \cite{GSO}.  For $n \in S$ we define $\rho_n$  as follows:
\begin{itemize}
\item If $\varphi^{-1}(n)$ has a single $\mR$-class, then $\rho_n = \emptyset$.
\item Otherwise, let $\mR_1,\ldots, \mR_k$ be the different $\mR$-classes of $\varphi^{-1}(n)$.  Choose some $v_i \in \mR_i$ for each $i \in [1,k]$ and set $\rho_n = \{(v_1,v_2),(v_2,v_3),\ldots, (v_{k-1},v_k)\}$. 
\end{itemize}
Then $\rho = \bigcup_{n \in S} \rho_n$ is a minimal presentation of $S$.  It is known that all minimal presentations of $S$ have the same cardinality, but they are only unique under an additional hypothesis. 

\begin{thm}[Corollary 6 in \cite{GSO}]
A numerical semigroup $S$ is uniquely presented if and only if every Betti element of $S$ has exactly two factorizations.
\end{thm}

Theorem \ref{MinPres} follows from giving the set of Betti elements of $S$ and the factorizations of each such element.

\begin{lemma}\label{Lem1}
Let $S$ be as in Theorem \ref{MinPres}.  The Betti elements of $S$ are $(2x(p-1)-1)(p^x-2)$ and $p(2(p^x-2)+p^{i-1})$ for each $i \in [1,x]$.
\end{lemma}
Before proving this lemma we show that the generating set given in Theorem \ref{MinPres} is minimal.  Suppose that this is not the case. Then there is a generator that is a nonnegative linear combination of smaller generators, which means that for some $i \in [1,x],\ 2(p^x-2) + p^{i-1}$ can be written as a sum of other generators.  Since $p, x \ge 2$ and at least one is greater than $2$, we see that $p^{i-1} \not\equiv 0 \pmod{p^x-2}$, so this linear combination must contain a term $2(p^x-2) + p^{j-1}$ for some $j < i$.  However, this implies that $p^{i-1} - p^{j-1}$ is a sum of the generators.  Since $p^{j-1}(p^{i-j} - 1) < p^x-2$, this is impossible. Therefore, we have a minimal generating set. 

We now describe the complete set of factorizations of each of the elements given in Lemma \ref{Lem1}.  For each of these elements we exhibit two factorizations in different $\mR$-classes and show that every other element has a single $\mR$-class.  In fact, when $p > 2$ each of these elements has exactly two factorizations, which proves the following.

\begin{cor}\label{Cor1}
Let $S$ be as in the statement of Theorem \ref{MinPres}.  Then $S$ is uniquely presented if and only if $p>2$.
\end{cor}

\begin{lemma}\label{Lem2}
For $i \in [1,x-1]$, the element $n = p(2(p^x-2) + p^{i-1})$ has exactly two factorizations and they are in different $\mR$-classes. 
\end{lemma}

\begin{proof}
Suppose $n = p(2(p^x-2) + p^{i-1})$ where $i \in [1,x-1]$. We check that 
\begin{equation}\label{TwoFactorizations}
n = 2(p-1)\cdot (p^x-2) + 1\cdot (2(p^x-2) + p^i) = p\cdot (2(p^x-2)+p^{i-1}).  
\end{equation}
This gives two factorizations of $n$ in different $\mR$-classes.  Suppose there is another factorization of $n,\ B = (B_0,B_1,\ldots, B_x)$.  Then since
\[
n = B_0 (p^x-2) + \left(\sum_{j=1}^x B_j \right) 2(p^x-2) + \left(\sum_{j=1}^x B_j p^{j-1}\right),
\]
we see that 
\begin{equation}\label{Bjmod}
\sum_{j=1}^x B_j p^{j-1} \equiv p^i \pmod{p^x-2}.
\end{equation}

Since $i \le x-1$ and $p^i \le p^{x-1} < p^x -2$, if $\sum_{j=1}^x B_j > p$ then
\[
n - \left(\sum_{j=1}^x B_j \right) 2(p^x-2) \le n - 2(p+1)(p^x-2) = p^i-2(p^x-2) < 0.
\] 
Therefore, $\sum_{j=1}^x B_j \le p$.  

If $\sum_{j=1}^x B_j = p$, then 
\[
n - \left(\sum_{j=1}^x B_j \right) 2(p^x-2) = p^i,
\]
and (\ref{Bjmod}) implies that $\sum_{j=1}^x B_j p^{j-1} = p^i$.   Since 
\[
(p-1) \left(1+p+p^2 +\cdots + p^{i-1}\right) = p^i-1,
\] 
the only way to write
\[
p^i = \sum_{j=1}^{x} a_j p^{j-1}
\]
with each $a_j \ge 0$ and $\sum_{j=1}^{x} a_j = p$ is to have $a_{i} = p$ and $a_j = 0$ for $j \neq i$.  This gives the second factorization of (\ref{TwoFactorizations}).

If $\sum_{j=1}^x B_j < p$ then $\sum_{j=1}^x B_j p^{j-1} < (p-1)p^{x-1} < p^x-2$ implies $\sum_{j=1}^x B_j p^{j-1} = p^i$.  The only way to write 
\[
p^i = \sum_{j=1}^{x} a_j p^{j-1}
\]
with each $a_j \ge 0$ and $\sum_{j=1}^{x} a_j < p$ is to have $a_{i+1} = 1$ and $a_j = 0$ for $j\neq i$.  This gives the first factorization of (\ref{TwoFactorizations}) and completes the proof of the lemma.
\end{proof}

\begin{lemma}\label{Lem3}
The element $n = p(2(p^x-2) + p^{x-1})$ has exactly two factorizations when $p>2$ and has exactly three factorizations when $p=2$.  In either case, the factorizations of this element belong to exactly two $\mR$-classes.
\end{lemma}

\begin{proof}
Let $n = p(2(p^x-2) + p^{x-1})$.  We check that 
\begin{equation}\label{TwoFactorizationsB}
n  = p\cdot (2(p^x-2)+p^{x-1})= (2p-3)\cdot (p^x-2) + 2 \cdot (2(p^x-2) + 1).
\end{equation}
If $p = 2$ we have one additional factorization,
\[
n = (2p-1)\cdot (p^x-2) + 1\cdot (2(p^x-2) + p).
\]  

Suppose that there is another factorization $B = (B_0,\ldots, B_x)$.  As in the proof of the previous lemma, we see that $\sum_{j=1}^x B_j p^{j-1} \equiv 2 \pmod{p^x-2}$. Since 
\[
n - (p+1)(2(p^x-2) + 1) = p^x - 2(p^x-2) = 4 - p^x < 0,
\]
we see that $\sum_{j=1}^x B_j \le p$. 

If $p = 2$ there are exactly three ways to add up at most two elements from $\{1,2,2^2,\ldots, 2^{x-1}\}$ to get something equivalent to $2$ modulo $2^x-2$. If $p>2$ there are exactly two ways to add up at most $p$ elements from $\{1,p,p^2,\ldots, p^{x-1}\}$ two get something equivalent to $2$ modulo $p^x-2$, since their sum must equal either $2$ or $p^x$.  Choosing such a set determines the factorization $B$ and we see that we have found all factorizations of $n$.
\end{proof}

In order to characterize the set of factorizations of the last Betti element of $S$ we prove a lemma that allows us to better understand the factorization in each $\mR$-class with the largest number of copies of $p^x-2$, the smallest generator of $S$.

\begin{lemma}\label{Lem4}
Suppose that $A = (A_0,A_1,\ldots, A_x)$ is a factorization of $n \in S$ with $A_i \ge p$ for some $i \in [1,x]$.  Then either $A$ is in the same $\mR$-class as a factorization $(B_0,\ldots, B_x)$ with $B_0 > A_0$, or $n = p(2(p^x-2) + p^{i-1})$.
\end{lemma}

\begin{proof}
Suppose that $n \neq  p(2(p^x-2) + p^{i-1})$ and that $A$ is a factorization of $n$ with $A_i \ge p$.  Since $n \neq p\cdot (2(p^x-2) + p^{i-1})$ either $A_i \ge p+1$ or $A_j \neq 0$ for some $j \neq i$.  

If $A_x \ge p$ then $A$ is in the same $\mR$-class as the factorization $(A_0 + 2p-3,A_1 +2,A_2,\ldots, A_{x-1}, A_x-p)$. If $A_i \ge p$ for some $i \in [1,x-1]$ then $A$ is in the same $\mR$-class as $(A_0 + 2p-2, A'_1,\ldots, A'_x)$ where $A'_i = A_i -p,\ A'_{i+1} = A_{i+1} + 1$, and $A'_j = A_j$ otherwise.  
\end{proof}
The following result shows that if a factorization contains too many copies of the smallest generator then the set of all factorizations of this element consists of a single $\mR$-class.

\begin{lemma}\label{Lem5}
Suppose that $A = (A_0,A_1,\ldots, A_x)$ is a factorization of $n \in S$ with $A_0 \ge 2x(p-1) -1$.  Then $\varphi^{-1}(n)$ consists of a single $\mR$-class, or $n = (2x(p-1) -1)(p^x-2)$.
\end{lemma}

\begin{proof}
As in the proof of the previous lemma, if $n \neq (2x(p-1) -1)(p^x-2)$ then either $A_0 > 2x(p-1)-1$ or there exists some $i \in [1,x]$ with $A_i > 0$.  Then 
\[
A' = (A_0 - (2x(p-1)-1), A_1 + p-2, A_2 + p-1,\ldots, A_x + p-1)
\] 
is another factorization in the same $\mR$-class as $A$.  Every factorization has support that intersects either $\supp(A)$ or $\supp(A')$.  For $n \neq (2x(p-1) -1)(p^x-2),\ \supp(A) \cap \supp(A') \neq \emptyset$ and these factorizations are in the same $\mR$-class.  Therefore, $\varphi^{-1}(n)$ consists of a single $\mR$-class.
\end{proof}

\begin{lemma}\label{Lem6}
The element $n = (2x(p-1) -1)(p^x-2)$ has exactly two factorizations and they are in different $\mR$-classes.
\end{lemma}

\begin{proof}
Suppose $n = (2x(p-1) -1)(p^x-2)$ and that $A = (A_0,\ldots, A_x)$ is a factorization of $n$.  Since 
\[
n = A_0 (p^x-2) + \left(\sum_{i=1}^x A_i\right) (p^x-2) +  \left(\sum_{i=1}^x A_i p^{i-1}\right),
\]
we see that $\sum_{i=1}^x A_i p^{i-1} \equiv 0 \pmod{p^x-2}$.  First suppose that for each $i \in [1,x]$ that $A_i < p$.  Then either $A_i = 0$ for all $i \in [1,x]$ or $A_1 = p-2$ and $A_2,\ldots, A_x = p-1$.  In the first case we get the factorization $(2x(p-1)-1,0,\ldots, 0)$ and in the second case we get $(0,p-2,p-1,\ldots, p-1)$.

Now suppose that there is some other factorization $(B_0,B_1,\ldots, B_x)$ where $B_i \ge p$ for some $i \in [1,x]$.  The proof of Lemma \ref{Lem4} shows that we can exchange $p$ copies of the generator $2(p^x-2) + p^{i-1}$ to get another factorization $C = (C_0,\ldots, C_x)$ with $C_0 > B_0$. If there is an $i \in [1,x]$ with $C_i \ge p$ we can again trade $p$ copies of a single generator to find another factorization with a larger number of copies of the first generator.  Eventually this process terminates since we have a larger number of copies of the first generator each time.  When it does, we get a factorization $D = (D_0,D_1,\ldots, D_x)$ that must have $D_j = 0$ for all $j \in [1,x]$ as $\sum_{j=1}^x D_j p^{j-1} \equiv 0 \pmod{p^x-2},\ D_j < p$ for each $j \in [1,x]$, and $D_0 > 0$.  However since we traded $p$ copies of a single generator $2(p^x-2)+p^{i-1}$ to go from the previous factorization to $D$, we see that 
\[
p\cdot (2(p^x-2)+p^{i-1}) \equiv p^i  \equiv 0 \pmod{p^x-2},
\]
which is impossible. Therefore, the only two factorizations are those listed above.
\end{proof}

We have characterized the set of factorizations of a special set of $x+1$ elements of $S$ and now show that the set of factorizations of any other element form a single $\mR$-class. This proves that we have found the Betti elements of $S$, which completes the proof of Theorem \ref{MinPres}.
\begin{proof}[Proof of Lemma \ref{Lem1}]
Suppose that $n \in S$ has at least two $\mR$-classes and that $n$ is not equal to $(2x(p-1)-1)(p^x-2)$ or $p(2(p^x-2)+p^{i-1})$ for any $i\in [1,x]$.

Suppose that $A = (A_0,A_1,\ldots, A_x)$ and $B= (B_0,B_1,\ldots, B_x)$ are factorizations of $n$ in distinct $\mR$-classes.  Then $A$ and $B$ have disjoint support.  We replace $A$ and $B$ with the factorizations in their $\mR$-classes with the largest values of $A_0$ and $B_0$.  Without loss of generality we can suppose that $B_0 = 0$.  By Lemma \ref{Lem4} we can suppose that $A_i, B_i < p$ for each $i \in [1,x]$. By Lemma \ref{Lem5} we can suppose that $A_0 < 2x(p-1)-1$.

Since $A$ and $B$ are factorizations of the same element and $B_0 = 0$, we have
\[
\left(A_0 + 2 \sum_{i=1}^x A_i\right) (p^x-2) + \sum_{i=1}^x A_i p^{i-1} = \left(2 \sum_{i=1}^x B_i\right) (p^x-2) + \sum_{i=1}^x B_i p^{i-1}.
\] 
This implies 
\[
\sum_{i=1}^x A_i p^{i-1} \equiv \sum_{i=1}^x B_i p^{i-1} \pmod{p^x-2}.
\]
Every integer $k \in [0,p^x - 1]$ can be written uniquely as 
\[
k = \sum_{i=1}^x a_i p^{i-1},
\]
where $0\le a_i < p$. The only way for $\sum_{i=1}^x B_i p^{i-1}$ to be at least $p^x-2$ is when $(B_1,\ldots, B_x) = (p-2,p-1,\ldots, p-1)$, or $(p-1,\ldots, p-1)$.  In the first case $\sum_{i=1}^x B_i p^{i-1}$ is congruent to $0$ modulo $p^x-2$, and in the second case it is congruent to $1$.  In both of these special cases the factorization $B$ is in the same $\mR$-class as a factorization with a larger number of copies of the smallest generator, which is a contradiction.  If $\sum_{i=1}^x B_i p^{i-1} \not\equiv 0,1 \pmod{p^x-2}$ then the values of $(B_1,\ldots, B_x)$ are completely determined and must be equal to $(A_1,\ldots, A_x)$, contradicting the assumption that these factorizations have disjoint support.
\end{proof}

Lemma \ref{Lem1} together with the complete set of factorizations of each of these elements proves Theorem \ref{MinPres}.  Now that we have a more detailed understanding of the factorizations of elements of $S$ we can easily compute $\Delta(S)$.

\begin{thm}\label{TwoEl}
Let $S$ be as in the statement of Theorem \ref{MinPres}.  Then 
\[
\Delta(S) = \{p-1,(p-1)x\}.
\] 
\end{thm}

We note that if $p=x=2$ then we get the non-minimal generating set for $S\ \{2,5,6\}$.  Following the conventions of \cite{CDHK}, the delta set of $\langle 2,5\rangle$ with respect to this generating set is $\{1,2\}$, which is consistent with Theorem \ref{TwoEl}.

\begin{proof}
By Proposition \ref{MinD} the minimum element of $\Delta(S)$ is equal to the greatest common divisor of the differences between consecutive minimal generators.  These differences are 
\[
\{p^x-1,p-1,p(p-1),\ldots, p^{x-2}(p-1)\},
\] 
a set with greatest common divisor equal to $p-1$.  Therefore, $\min \Delta(S) = p-1$.  The characterization of the factorizations of $(2x(p-1)-1)(p^x-2)$ given in Lemma \ref{Lem6} shows that $x(p-1)\in \Delta(S)$.

Therefore, if $A = (A_0,\ldots, A_x)$ and $B = (B_0,\ldots, B_x)$ are two factorizations of the same element of $S$, then 
\[
|A - B| := \left| \sum_{i=0}^x (A_i - B_i) \right| = k \cdot (p-1),
\]
for some $k \ge 0$.  Throughout the rest of the proof we suppose that $A$ and $B$ are two factorizations of the same element in $S$ with $\sum_{i=0}^x A_i > \sum_{i=0}^x B_i$ and that there are no factorizations with length in between these two.  We will show that if $|A-B| > p-1$ then $|A-B| = x(p-1)$, completing the proof.

We argue by contradiction.  Suppose that $|A-B| = k(p-1)$ with $k \in [2,x-1]$.  If such a pair of factorizations exists, then by canceling common elements there exists such a pair where for each $i$ either $A_i =0$ or $B_i = 0$.   We first show that we can make simplifying assumptions about the $A_i$ and $B_i$ by showing that if these assumptions are not satisfied then we can either find a factorization of length exactly $p-1$ longer than the length of $B$ or exactly $p-1$ shorter than the length of $A$.  

If $B_x \ge p$, then $(B_0 + 2p-3,B_1+2,B_2,\ldots, B_{x-1}, B_x-p)$, is a factorization of the same element with length exactly $p-1$ longer than the length of $B$, which is a contradiction.  If $B_i \ge p$ for some $i\in [1,x-1]$ then $(B_0 +2p-2,B'_1,B'_2,\ldots, B'_x),$ where $B'_i = B_i-p,\ B'_{i+1} = B_{i+1}+1$, and $B'_j = B_j$ otherwise, is a factorization of the same element with length exactly $p-1$ longer than the length of $B$, which is a contradiction.

We have
\[
(A_0-B_0)(p^x-2) + \sum_{i=1}^x (A_i - B_i) \left(2 (p^x-2) +p^{i-1}\right) = 0.
\]
This gives
\[
A_0 - B_0 + 2 \sum_{i=1}^x (A_i - B_i) + \frac{\sum_{i=1}^x (A_i - B_i)p^{i-1}}{p^x-2} = 0,
\]
which implies
\begin{equation}\label{Eqn4}
2k(p-1) + \frac{\sum_{i=1}^x (A_i - B_i)p^{i-1}}{p^x-2} = A_0 - B_0.
\end{equation}
Since $B_i \le p-1$ for $i\in [1,x]$ we have 
\[
 \frac{\sum_{i=1}^x (A_i - B_i)p^{i-1}}{p^x-2}  \ge \frac{-(p-1)(1+p+\cdots+p^{x-1}) }{p^x-2} = -1 - \frac{1}{p^x-2},
\]
which implies that this sum is at least $-1$, as it is an integer.  So by (\ref{Eqn4})
\[
2k(p-1) -1 \le A_0 - B_0 \le A_0.
\]  
Since $k \ge 2$ and $p\ge 2$ we see that $A_0 \ge 4p-5 \ge 2p-2$, and conclude that $B_0 = 0$ by cancellation.

Since $A_0 \ge 2p-2$, if $A_i \ge 1$ for any $i \in [2,x]$ then $(A_0 - 2(p-1),A'_1,\ldots, A'_x)$ where $A'_{i-1} = A_{i-1}+p,\ A'_i = A'_i -1$, and $A'_j = A_j$ otherwise, gives another factorization of the same element with length exactly $p-1$ shorter than the length of $A$, which is a contradiction.  If $A_1 \ge 2$ then $(A_0 - (2p-3),A_1 -2,A_3, \ldots, A_{x-1}, A_x+p)$ is a factorization of the same element with length exactly $p-1$ shorter than the length of $A$, which is a contradiction.  Therefore we can suppose that $A_1 \le 1$ and $A_i = 0$ for all $i \in [2,x]$.  

Note that since $B_0 = 0$,
\[
\sum_{i=0}^x (A_i - B_i) = A_ 0 + (A_1 - B_1) - \sum_{i=2}^x B_i = k(p-1),
\]
and 
\[
2k(p-1) +\frac{A_1-B_1 - \sum_{i=2}^x B_i p^{i-1}}{p^x-2} = A_0
\]
by (\ref{Eqn4}).

Since $p \ge 2$, $A_1 \le 1$, and at least one $B_i \ge 1$, in order for the fraction to be an integer we must have 
\[
A_1 - \sum_{i=1}^x B_i p^{i-1}= -(p^x-2) t
\]
for some positive integer $t$. Since each $B_i < p$ and $A_1\le 1$ we must have $t=1$. Since at least one of $A_1, B_1$ equals zero, we see that $A_1 = 0$ and
\[
(B_0, B_1, \ldots, B_x) = (0,p-2,p-1,\ldots, p-1).
\]
Since $A_i = 0$ for all $i \in [1,x]$ we have $A_0 = \frac{n}{p^x-2} = 2x(p-1)-1$.  This gives $|A-B| = (p-1)x$, which contradicts the assumption that $|A-B| < (p-1)x$. 
\end{proof}

\section{Two-Element Delta Sets of Semigroups with Embedding Dimension Three}\label{3gen}

In this section we characterize precisely which two-element sets occur as the delta set of some numerical semigroup of embedding dimension three.  By Proposition \ref{Geroldinger}, such a set must be of the form $\{d,td\}$ for some positive integers $d \ge 1$ and $t\ge 2$.  We show that $\{d,td\}$ occurs as a delta set if and only if $t = 2$.

\begin{thm}\label{no13}
Suppose that $S = \langle n_1, n_2, n_3\rangle$.  Let $d = \gcd\{n_3 - n_2, n_2 - n_1\}$.  Then $|\Delta(S)| = 2$ implies that $\Delta(S) = \{d,2d\}$.  
\end{thm}

The main tool in this argument is a careful consideration of the minimal presentations of embedding dimension three numerical semigroups.  We follow the presentation in Section 4 of \cite{Chap}.  A numerical semigroup $S$ must have $|\N\setminus S|<\infty$.  The largest element of $\N \setminus S$ is called the \emph{Frobenius number}, and is denoted $F(S)$.  A semigroup $S$ is \emph{symmetric} if for each $i \in [1,F(S)]$ exactly one of $\{i, F(S)-i\}$ is in $S$.  There are two cases to consider based on whether $S$ is a symmetric or not.

\begin{proof}
Let $S = \langle n_1, n_2, n_3 \rangle$ with $n_1 < n_2 < n_3$ be a numerical semigroup of embedding dimension three that is not symmetric. We recall some facts from \cite{Chap} that are also covered in detail in Chapter 9 of \cite{GSR}. For $i \in [1,3]$ there exist positive integers $r_{ij}$ such that 
\[
c_i n_i = r_{ij} n_j + r_{ik} n_k,
\]
where $c_i = \min\{k\in \N \setminus \{0\}\ |\ kn_i \in \langle n_j,n_k\rangle\}$.  There is a unique minimal presentation of $S$ given by 
\[
\sigma = \{((c_1,0,0),(0,r_{12},r_{13})),((0,c_2,0),(r_{21},0,r_{23})),((0,0,c_3),(r_{31},r_{32},0))\}.
\]
We note that
\[
(c_1,-r_{12},-r_{13}) + (-r_{21},c_2,-r_{23}) + (-r_{31},-r_{32},c_3) = (0,0,0).
\]

The three elements $c_1 n_1, c_2 n_2, c_3 n_3$ are distinct and each has exactly two factorizations.  Let $\delta_1 = c_1 - (r_{12} + r_{13}),\ \delta_3 = (r_{31} + r_{32}) - c_3$, and $\delta_2 = |c_2 - (r_{21} + r_{23})|$.  Since $n_1 < n_2 < n_3$ we see that $\delta_1, \delta_3 > 0$, and that $\delta_1, \delta_3 \in \Delta(S)$, and $\delta_2\in \Delta(S)$ if it is nonzero.  Moreover, Corollary 3.1 of \cite{Chap} implies that $\max \Delta(S) = \max\{\delta_1,\delta_3\}$ and that each element of $\Delta(S)$ can be written as 
\[
\lambda_1 \delta_1 + \lambda_3 \delta_3
\]
for some $\lambda_1, \lambda_3 \in \Z$.  If $|\Delta(S)| > 1$ then $\delta_1 \neq \delta_3$.  We also have that $\delta_2 = |\delta_1 - \delta_3|$.  Suppose that $\Delta(S) = \{d,td\}$ with $t > 2$.  Then $\{\delta_1,\delta_3\} = \{d,td\}$ and $\delta_2 = (t-1)d$, which is a contradiction.

We now consider the case where $S$ is a symmetric numerical semigroup of embedding dimension three, closely following the presentation of Section 4.3 of \cite{Chap}.  Theorem 10.6 in \cite{GSR} implies that $S = \langle a m_1, a m_2, b m_1 + c m_2\rangle$ for some nonnegative integers $m_1,m_2, a,b,c$ satisfying $a,b+c \ge 2$ and $\gcd\{m_1, m_2\} = 1 = \gcd\{a,bm_1 + cm_2\}$.  Without loss of generality suppose $m_2 > m_1$.  Theorem 17 in \cite{GSO} implies that a minimal presentation of $S$ is 
\[
\sigma = \{((m_2,0,0),(0,m_1,0)), ((0,0,a),(b,c,0))\}.
\]
This presentation is not necessarily unique.

We see that the element $a m_1 m_2 \in S$ has exactly two factorizations, so $m_2 - m_1 \in \Delta(S)$.  Let $r = \left\lfloor \frac{c}{m_1} \right\rfloor$ and $s= \left\lfloor\frac{b}{m_2} \right\rfloor$.  In Section 4.3 of \cite{Chap} the authors show that the set of lengths of $a(bm_1 + c m_2)$ is given by
\begin{eqnarray*}
\{a,b+c - s(m_2-m_1),b+c - (s-1)(m_2-m_1),\ldots  \\
\ldots, b+c + (r-1) (m_2-m_1), b+c + r(m_2-m_1)\},
\end{eqnarray*}
and that $|\Delta(S)| = 1$ if and only if $a = b+c +k(m_2-m_1)$ for some integer $k \in [-s-1,r+1]$.  Assume that this is not the case.

Suppose that $\Delta(S) = \{d,td\}$ with $t > 2$.  By assumption, $a$ is not equal to $b+c + k(m_2 - m_1)$ for any $k \in [-s-1,r+1]$. If 
\[
b+c - s(m_2-m_1) < a < b+c + r(m_2-m_1),
\]
then there exists $k \in [-s,r-1]$ such that 
\[
b+c + k (m_2-m_1) < a < b+c + (k+1) (m_2-m_1).
\]  
Taking differences shows that 
\[
\left\{b+c + (k+1) (m_2-m_1) -a, a - (b+c +k(m_2-m_1)) \right\} \subseteq \Delta\left(a(bm_1 + c m_2)\right).
\] 
These elements are both smaller than $m_2-m_1$, so if $\Delta(S) = \{d,td\}$ then $m_2 - m_1 = td$.  However, if both of these elements are equal to $d$ then $t=2$, which is a contradiction.  Therefore, $|\Delta(S)| \ge 3$, which is a contradiction.

We now consider two final cases.  First suppose that $a < b+c - s(m_2-m_1)$ and recall that $a \neq b+c -(s+1) (m_2-m_1)$.  Let $d_1 = m_2 - m_1$ and $d_2 = b+c - s(m_2 - m_1) - a$, and note that $d_1, d_2 \in \Delta(S)$ and $d_1 \neq d_2$.  Now consider the set of lengths of the element $a(b m_1 + c m_2) + a m_1 m_2$.  This element has factorizations $(m_2, 0, a)$ and $(b - s m_2, c + sm_1 + m_1,0)$, and no factorizations of lengths in between.  We conclude that 
\[
\left|b+c - s(m_2 - m_1) - a - (m_2 - m_1)\right| = |d_2 - d_1| \in \Delta(a(b m_1 + c m_2) + a m_1 m_2).
\]
Since $\{d_1, d_2\} = \{d,td\}$ we conclude that $(t-1)d \in \Delta(S)$, which is a contradiction.

The other case to consider is when $a > b+c + r(m_2 - m_1)$. A version of exactly the same argument with $-s$ replaced by $r$ again shows that $(t-1)d\in \Delta(S)$.  This completes the proof.
\end{proof}

We now show that for each $d \ge 1$, there is a numerical semigroup $S = \langle n_1, n_2, n_3\rangle$ with $\Delta(S) = \{d,2d\}$.  In fact, there is a symmetric semigroup of this type.
\begin{prop}
Let $d\ge 1$ be a positive integer and $p$ be an odd prime that does not divide $d$.  Let $S = \left\langle p^2, p^2+2d, p^2 - (p-2)d \right\rangle$.  Then $\Delta(S) = \{d,2d\}$.
\end{prop}

\begin{proof}
We verify that $S$ is a symmetric numerical semigroup. It is of the form $\langle am_1, am_2, b m_1 + c m_2\rangle$ with $a= m_1 = p,\ m_2 = p+2d,\ b = p-d-1$, and $c=1$.  Clearly $\gcd\{p, p+2d\} = 1$ since $\gcd\{p,2d\} = 1$.  Also, $\gcd\{p,p^2 -(p-2)d\} = 1$ since $\gcd\{p,(p-2)d\} = 1$.

We compute
\begin{eqnarray*}
\min \Delta(S) &  =  & \gcd \Delta(S) =  \gcd\{p^2+2d - p^2, p^2 - (p^2-(p-2)d) \} \\
& = &  \gcd\{2d, (p-2)d\} = d,
\end{eqnarray*}
since $p$ is an odd prime.  By Theorem 17 of \cite{GSO}, this semigroup has a unique minimal presentation since $0 < b < m_2$ and $0 < c < m_1$.  The Betti elements are $m_1 m_2$ and $b m_1 + c m_2$, which give delta set elements $m_2 - m_1$ and $a-(b+c)$. We see that 
\[
\max \Delta(S) = \max\{m_2 - m_1, a- (b+c)\} = \max\{2d,d\},
\] 
completing the proof.
\end{proof}

\section{Further Questions}

In this section we first suggest many other classes of semigroups with interesting delta sets  that we can describe explicitly.  We then give two related realization problems for numerical semigroups. Extensive computation suggests many other statements analogous to Theorem \ref{TwoEl}.

\begin{conj} Let m $\geq$ 1 be a positive integer and k $\geq$  0 be nonnegative integer.  Let 
\[
S = \langle 3 \cdot 2^{m+k} - 2^m, 2( 3 \cdot 2^{m+k}-2^m)+1,\ldots,2(3 \cdot 2^{m+k}-2^m)+2^m \rangle.
\]
Then $S$ is a complete-intersection numerical semigroup with minimal presentation given by 
\[
((3 \cdot 2^{k+1}-1,0,\ldots,0),(0,\ldots,0,3 \cdot 2^k-1))
\]
and for each $i \in [1,m]$ 
\[
v_i := ((0,\ldots,0,2,0,\ldots,0),(2,0,\ldots,0,1,0,\ldots,0))
\]
where the first $2$ is in the $i+1$ position, where we count starting at $0$, and the $1$ is in the $i+2$ position.  We also have that 
\[
\Delta(S) = \begin{cases}
\{1,\ldots,3 \cdot 2^{k}\} & \text{ if }   m=1,\\ 
\{1, \ldots,3 \cdot 2^{k}\} \setminus (\{3 \cdot 2^{k}+1-3k | k \in \N \} \setminus{1}) & \text{ if }  m=2, \\ 
\{1, \ldots, 3 \cdot 2^{k} \} \setminus (\{ 3 \cdot 2^{k} + n - 7k | k \in \N, n=1,2,5 \}\setminus{1}) & \text{ if }  m \geq 3. 
\end{cases}
\]
\end{conj}
Using the same argument as in the proof of Lemma \ref{Lem1}, we see the given generating set of $S$ is always minimal as $2^i-2^j<3\cdot2^{m+k}-2^m$ for the given parameters.

It would be very interesting to compare this class of complete intersection numerical semigroups to those given in \cite{AGS, DMS}. Slight variations of the semigroups given above seem to give interesting delta sets.

\begin{conj}\label{Con3}
For $x \ge 2$ if $S =  \langle 2^x, 2 \cdot 2^x +1,\ldots,2 \cdot 2^x + 2^{x-1} \rangle$, then 
\[ 
\Delta(S) = \{1,2,3\}.
\]
For $x \ge 3 $ if $ S = \langle 2^{x-1}-1, 2(2^{x-1}-1)+1,\ldots,2(2^{x-1}-1)+2^{x-2}\rangle $, then 
\[ 
\Delta(S) = \{1,2,x\}.
\]
For $x \ge 2 $, if $ S = \langle 2^{x+1} - 3, 2(2^{x+1}-3)+1,\ldots,2(2^{x+1}-3)+2^x \rangle $, then 
\[ 
\Delta(S) = \{1,x,x+1\}. 
\]

For $c \ge 4 $ and $x \ge 2$, if 
\[
S = \langle 2^{x+c-3}-c,2(2^{x+c-3}-c)+1,\ldots,2(2^{x+c-3}-c)+2^{x+c-5} \rangle,
\] 
then 
\[ 
\Delta(S) = \{1,x+i_0,x+i_1,\ldots, x + i_{ \left \lfloor (c-1)/2 \right \rfloor} \}, 
\]
where $i_0 = 0 $, and $i_j = 
\begin{cases} 
i_{j-1}+1 & \text{ if } j \equiv 1 \pmod{2} \\
i_{j-1}+2 & \text{ if } j \equiv 2 \pmod{4} \\
i_{j-1}+3 & \text{ if } j \equiv 0 \pmod{4} \\
\end{cases} $.
\end{conj}
Using the same argument as in Lemma \ref{Lem1}, we see $S$ is minimal in each of these cases. Note that these semigroups are given by a construction very similar to the one from Theorem \ref{TwoEl}.  All of these semigroups are of the general form 
\[
S = \langle 2^x-c,2(2^x-c)+1,2(2^x-c)+2^{x+h} \rangle,
\]
where $c \ge 1$ and $h \ge 0$. Further generalizations give several other explicit classes of delta sets.

\begin{conj}
For any fixed $c, h \ge 0$ and for each $n\ge 2$ let  
\[
S_n = \langle 2^x-c,n(2^x-c)+1, 2^x-c,n(2^x-c)+2, \ldots,n(2^x-c)+2^{x+h}\rangle.
\]  
Suppose $\Delta(S_2) = \{1,c_0,c_1,\ldots,c_k \}$. Then $\Delta(S_n)$ is
\[ 
\{1,\ldots,n-1\} \bigcup \{(n-1)(c_0-1)+1,(n-1)(c_1-1)+1,\ldots,(n-1)(c_k-1)+1 \}.
\]

The last equation of Conjecture \ref{Con3} can be generalized as follows. For $c \ge 4,\ x \ge 2$, and $n\ge 2$, if 
\[
S = \langle 2^{x+c-3}-c,n(2^{x+c-3}-c)+1,\ldots,n(2^{x+c-3}-c)+2^{x+c-5}\rangle,
\] 
then 
\[ 
\Delta(S) = \{1,\ldots,n-1,(n-1)(x+i_0-1)+1,\ldots,(n-1)( x + i_{ \left \lfloor (c-1)/2 \right \rfloor} - 1)+1 \} 
\] 
with the integers $i_j$ defined as above.  
\end{conj}
As above, for the parameters given here the generating set of $S$ is minimal.

It seems likely that these conjectures can be generalized further.  For example, the formula above gives 
\[
\Delta(\langle 2^{x+1}-4,2(2^{x+1}-4)+1,\ldots,2(2^{x+1}-4)+2^{x-1} \rangle)=\{1,x,x+1\},
\] 
for $x \ge 2$.  Computation suggests that for $x\ge 7$, removing the last element gives 
\[
\Delta(\langle 2^{x-2}-4,2(2^{x-2}-4)+1,\ldots,2(2^{x-2}-4)+2^{x-5} \rangle) = \{1,2,4,5,x,x+1\}.
\]

We would like to understand whether every finite subset containing $1$ occurs as a delta set.  The only finite set containing $1$ with maximum element at most $5$ that we have not yet found is $\{1,3,4,5\}$.  We have also performed extensive computations in an attempt to find a numerical semigroup $S$ with $\Delta(S) = \{1,3,6\}$ but have not yet been successful.

We end this paper by describing two related realization problems.  We have focused so far on computing delta sets of numerical semigroups, a measure of the complexity of the structures of all of the sets of lengths of the infinite set of elements of the semigroup.  It is also interesting to ask finer questions about sets of lengths of individual elements of a semigroup.  
\begin{question}
\begin{enumerate}
\item Which finite sets occur as $\Delta(x)$ for some element $x$ in some numerical semigroup $S$?
\item Which finite sets occur as $\mathcal{L}(x)$ for some element $x$ in some numerical semigroup $S$?
\end{enumerate}
\end{question}

There are sets that have a positive answer to this first question that do not occur as $\Delta(S)$ for any semigroup $S$.  For example, in the semigroup $S = \langle 4,9,11 \rangle$ the element $36$ has set of lengths equal to $\{4,6,9\}$ and therefore has delta set equal to $\{2,3\}$.  Since it is not true that the minimum element of $\Delta(x)$ must be equal to the greatest common divisor of $\Delta(x)$ there are no obvious restrictions on sets that have a positive answer to this first question.  It is also not clear that every set that occurs as $\Delta(S)$ for some semigroup $S$ will also occur as $\Delta(x)$ for some individual element.

The second question was suggested by Alfred Geroldinger.  It is clear that an element of a semigroup $S$ has a factorization of length $1$ if and only if it is a minimal generator, and in that case there is a unique factorization of this element.  However, there are no obvious restrictions on sets not containing $1$ to have a positive answer to this second question.  Sets of lengths within a given semigroup are known to satisfy certain structural conditions.  Similar realizations questions for sets of lengths have been considered by Schmid in other settings \cite{Schmid}.  It is likely that Geroldinger's structure theorem for sets of lengths, \cite{G, GHK}, will be a useful starting place for studying these questions.

\section{Acknowledgments}
The authors thank Christopher O'Neill for sharing helpful programs for computing delta sets, and Pedro Garc\'ia-S\'anchez for alerting them to the reference \cite{GSLM}.  The second author thanks Alfred Geroldinger for helpful correspondence related to the final section of the paper.  The authors also thank Gilana Reiss and the Science Research Mentorship Program at Hunter College High School.  Finally, the authors thank the referee for several comments that helped improve the paper.

\end{document}